\documentclass{amsart}

\usepackage{enumerate}
\usepackage{amssymb}
\usepackage{amsmath,amscd}
\usepackage{amsthm}

\newcommand{\R}{\mathbb{R}}
\newcommand{\C}{\mathbb{C}}
\newcommand{\HH}{\mathbb{H}}

\newcommand{\Ric}{\operatorname{Ric}}
\newcommand{\II}{I\!I}
\newcommand{\N}{\mathcal{N}}

\newcommand{\g}{\mathfrak{g}}
\newcommand{\kk}{\mathfrak{k}}
\newcommand{\m}{\mathfrak{m}}

\newcommand{\ad}{\operatorname{ad}}
\newcommand{\Ad}{\operatorname{Ad}}
\newcommand{\tr}{\operatorname{tr}}
\newcommand{\Ind}{\operatorname{ind}}
\newcommand{\ACS}{\operatorname{ACS}}

\newcommand{\OO}{\operatorname{O}}
\newcommand{\SO}{\operatorname{SO}}
\newcommand{\U}{\operatorname{U}}
\newcommand{\SU}{\operatorname{SU}}
\newcommand{\Sp}{\operatorname{Sp}}

\newtheorem{theorem}{Theorem}

\newtheorem*{conjecture}{Conjecture}
\newtheorem{corollary}[theorem]{Corollary}
\newtheorem{lemma}[theorem]{Lemma}
\newtheorem{proposition}[theorem]{Proposition}

\newtheorem{maintheorem}{Theorem}

\theoremstyle{definition}
\newtheorem{definition}[theorem]{Definition}

\theoremstyle{remark}
\newtheorem{remark}[theorem]{Remark}

\title[Robust index bounds]{Robust index bounds for minimal hypersurfaces of isoparametric submanifolds and symmetric spaces}

\author[C.~Gorodski]{Claudio Gorodski$^\ddagger$$^\mathsection$}
\address{University of S\~ao Paulo, Brazil}
\email{gorodski@ime.usp.br}

\author[R.~Mendes]{Ricardo A. E. Mendes$^*$$^\dag$$^\mathsection$}
\address{University of Cologne, Germany}
\email{rmendes@gmail.com}

\author[M.~Radeschi]{Marco Radeschi$^*$}
\address{University of Notre Dame, US}
\email{mradesch@nd.edu}

\thanks{$^\ddagger$ received support from CNPq grant 303038/2013-6 and the FAPESP project 16/23746-6}
\thanks{$^*$ received support from SFB 878: \emph{Groups, Geometry \& Actions}}
\thanks{$^\dag$ received support from DFG ME 4801/1-1}
\thanks{$^\mathsection$ received support from DFG SFB TRR 191}


\begin{document}

\begin{abstract}
We find many examples of compact Riemannian manifolds $(M,g)$ whose closed minimal hypersurfaces satisfy a lower bound on their index that is linear in their first Betti number. Moreover, we show that these bounds remain valid when the metric $g$ is replaced with $g'$ in a neighbourhood of $g$. Our examples $(M,g)$ consist of certain minimal isoparametric hypersurfaces of spheres; their focal manifolds; the Lie groups $\SU(n)$ for $n\leq 17$, and $\Sp(n)$ for all $n$; and all quaternionic Grassmannians.
 \end{abstract}

\maketitle


\section{Introduction}

Let $(M,g)$ be a compact Riemannian manifold. We are interested in closed, immersed, minimal hypersurfaces $\Sigma\to M$. The (Morse) index $\Ind(\Sigma)$ of such  $\Sigma$ is the  maximal dimension of a space of smooth sections of the normal bundle of $\Sigma$ where the second variation of area is negative-definite. Since $\Sigma$ is compact, its index is finite. It is natural to ask what is the relation between the index and the topology of $\Sigma$, and one conjecture in this regard is the following (see \cite[page 16]{Neves14}, or \cite[page 3]{ACS16}):
\begin{conjecture}[Marques-Neves-Schoen]
\label{conjecture}
Let $(M,g)$ be a compact Riemannian manifold with positive Ricci curvature, and dimension at least three. Then there exists $C>0$ such that, for all closed, embedded, orientable, minimal hypersurfaces $\Sigma\to M$, one has
\[\Ind(\Sigma)\geq Cb_1(\Sigma)\]
where $b_1(\Sigma)$ denotes the first Betti number of $\Sigma$ with real coefficients.
\end{conjecture}

Some special cases and related results include: When $(M,g)$ is a flat $3$-torus, Ros \cite[Theorem 16]{Ros06} has found affine (in the first Betti number) bounds on the index --- see also \cite{ChodoshMaximo16} and \cite{ACS17}. For $(M,g)$ a round sphere of any dimension, Savo \cite{Savo10} has given linear bounds on the index. Generalizing Savo's method, Ambrozio, Carlotto, and Sharp \cite{ACS16} have found linear bounds on the index when $(M,g)$ is any compact rank-one symmetric space, or $S^a\times S^b$ for $(a,b)\neq (2,2)$. Finally, the second and third authors of the present article have proven a linear bound on the index plus nullity when $(M,g)$ is any compact symmetric space \cite{MendesRadeschi17}.

Savo's method, as generalized by Ambrozio-Carlotto-Sharp, relies on the existence of an isometric immersion of $(M,g)$ into a Euclidean space $\R^d$ such that a certain real-valued function, which we call the \emph{$\ACS$ quantity}, is everywhere negative. The domain of this function is the total space of  the bundle of Stiefel manifolds $V_2(TM)$ of orthonormal $2$-frames, and it depends only on the second fundamental form $\II$ of $M\to \R^d$, see Definition \ref{D:ACSqty} below for the precise formula. For instance, the standard inclusion $S^n\subset \R^{n+1}$ satisfies $\ACS<0$. It was already recognized in  \cite{ACS16} that this method is flexible in the sense that sometimes the obtained index bound remains valid when the ambient metric g is deformed in certain directions (see \cite[Theorems 12 and 13]{ACS16}). We push this idea to its natural limit and obtain:
\begin{maintheorem}
\label{MT:robust}
Suppose $(M,g)$ admits a $C^\infty$ isometric immersion into $\R^d$ with negative ACS quantity, and with image contained in a sphere. Let $\lambda\in (0,1)$. Then there exists $\epsilon>0$ such that: For any $C^\infty$ metric $g'$ on $M$ with $\|g-g'\|_{C^{2,\lambda}}<\epsilon$ (H\"older norm), and any minimal, closed, immersed hypersurface $\Sigma\subset (M,g')$, one has
\[\Ind(\Sigma)\geq \frac{8}{d(d+3)(d^2+3d-2)} b_1(\Sigma).\]
\end{maintheorem}
Compared to the extrinsic flexibility of the method in  \cite{ACS16}, Theorem \ref{MT:robust} states that, when the image of the immersion $M\to\R^d$ is contained in a sphere, then the method is actually \emph{intrinsically} flexible: that is, the linear index bound remains valid under any small deformation of the metric itself.

The proof of Theorem \ref{MT:robust} is based on the proof of the Nash Embedding Theorem. The hypothesis that the image of $M\to \R^d$ is contained in a sphere is satisfied by all known examples (including our new examples described below) of immersions with negative $\ACS$ quantity. If one drops this hypothesis from the statement of Theorem \ref{MT:robust}, our proof  still yields an open set $\mathcal{U}$ of metrics on $M$ with respect to which the stated index bound holds. Moreover, $\mathcal{U}$ can be taken so that the original metric $g$ belongs to the closure of $\mathcal{U}$.

Among the ambient symmetric spaces mentioned earlier, Theorem \ref{MT:robust} applies to $S^n$, $S^a\times S^b$ for $(a,b)\neq (2,2)$, $\HH P^n$, and the Cayley plane. It does not apply to $\R P^n$ and $\C P^n$, because the proof in \cite{ACS16} of the index bound in this cases is less direct, and in particular they do not produce an immersion of these spaces into Euclidean space with $\ACS<0$.

In the second part of the present paper, we drastically expand the list of spaces to which Theorem \ref{MT:robust} may be applied, and for which, in particular, the conclusion of Conjecture holds. An interesting feature of our new examples is that they all have positive Ricci curvature (see Propositions \ref{P:Ricci} and \ref{P:focalRic}), while it is possible to find a sequence with  minimum  sectional curvature tending to $-\infty$ (see Proposition \ref{P:sec}), thus providing evidence that the curvature assumption in Conjecture is the correct one. Another novel feature is that infinitely many are not symmetric, and in fact not even homogeneous, although they are all \emph{curvature-homogeneous} (see Remark \ref{R:homogeneous}).

Our first class of examples are among \emph{isoparametric} submanifolds of Euclidean space, that is, submanifolds with flat normal normal bundle, and constant principal curvatures along any parallel section of the normal bundle; and their focal manifolds. Our motivation to consider isoparametric submanifolds is that they have the simplest extrinsic geometry, making the study of the $\ACS$ quantity more manageable. One interesting fact about such submanifolds is that, assuming the multiplicities are bigger than one, $\ACS<0$  \emph{implies} $\Ric>0$, see Proposition \ref{P:Ricci}.

Our concrete examples belong to the even more special class of isoparametric hypersurfaces of the sphere. These have been studied for at least a century by many prominent geometers, notably E. Cartan (see subsection \ref{SS:preliminaries} below for a short summary, and \cite{BCO} for a general reference). Nevertheless, they remain a very active area of research, with interesting questions still open.

\begin{maintheorem}
\label{MT:isoparametric}
Let $M^n\subset S^{n+1}$ be a minimal isoparametric hypersurface with four principal curvatures, and multiplicities $m_1\leq m_2$. Let $M_+$ be the focal manifold of $M$ with codimension $1+m_1$ in $S^{n+1}$. 
\begin{enumerate}[a)]
\item If $m_1\geq 5$, or if $m_1=4$ and $m_2$ is large enough, then $M$ satisfies $\ACS<0$. 
\item If $m_2>(3m_1+10)/4$, then  $M_+$ satisfies $\ACS<0$.
\end{enumerate}
\end{maintheorem}

There exist infinitely many homogeneous and inhomogeneous isoparametric hypersurfaces $M\subset S^{n+1}$ satisfying the conditions in (a) and (b) of Theorem \ref{MT:isoparametric} (see subsection \ref{SS:examples} for precise statements). The homogeneous spaces satisfying (a) are orbits of the group $\Sp(k)\Sp(2)$ acting on the space of quaternionic $k\times 2$ matrices in the natural way; as well as the isotropy representation of the symmetric space $E_6/\operatorname{Spin}(10)\U(1)$. The homogeneous focal manifolds satisfying (b) are Stiefel manifolds of $2$-frames over $\R$, $\C$, or $\HH$; and one of the singular orbits of the isotropy representation of the symmetric space $E_6/\operatorname{Spin}(10)\U(1)$. The inhomogeneous examples satisfying the conditions of Theorem \ref{MT:isoparametric} were constructed using Clifford systems by Ferus-Karcher-M\"unzner \cite{FerusKarcherMunzner81}, generalizing previous constructions of Ozeki-Takeuchi \cite{OzekiTakeuchi75,OzekiTakeuchi76}.

Our second class of new examples are symmetric spaces:
\begin{maintheorem}
\label{MT:symmetric}
The following symmetric spaces admit an embedding into some Euclidean space with $ACS<0$:
\begin{enumerate}[a)]
\item The quaternionic Grassmannian of $d$-planes in $\HH^n$, for all $d,n$;
\item The Lie group $Sp(n)$ for all $n$;
\item The Lie group $SU(n)$ for $n\leq 17$.
\end{enumerate}
\end{maintheorem}

This article is organized as follows. In Section \ref{S:indexbounds}, we recall the method of Savo and Ambrozio-Carlotto-Sharp to prove index bounds using an isometric immersion of the ambient manifold into Euclidean space, and in particular define what we call the $\ACS$ quantity. Section \ref{S:robust} is devoted to the proof of Theorem \ref{MT:robust}. Section \ref{S:isoparametric} concerns isoparametric hypersurfaces of the sphere. After some preliminaries, we compute the $\ACS$ quantity of such submanifolds, and prove Theorem \ref{MT:isoparametric}. Then, we apply Theorem \ref{MT:isoparametric} to concrete examples, and finish the section with remarks about the geometry of these examples. Finally, in Section \ref{S:symmetric} we study equivariant embeddings of symmetric spaces into Euclidean space, and prove Theorem \ref{MT:symmetric}: parts (a), (b), and (c) follow from Propositions \ref{P:quatGrass}, \ref{P:Sp(n)}, and \ref{P:SU(n)}, respectively.

\subsection*{Acknowledgements} It is a pleasure to thank Lucas Ambrozio and Alessandro Carlotto for useful discussions, and Alexander Lytchak and the University of Cologne for the hospitality during the visits of the first- and third-named authors. 

\section{Index bounds}
\label{S:indexbounds}
In this section we recall a method due to Ambrozio-Carlotto-Sharp \cite{ACS16} (generalizing previous work, especially \cite{Ros06} and \cite{Savo10}) to prove lower bounds on the index of immersed minimal hypersurfaces.

Consider a complete Riemannian manifold $(M,g)$. Assume $M$ is isometrically immersed into some Euclidean space $\R^d$, and denote by $\II$ the second fundamental form of this immersion. We define the following quantity as in \cite[Proposition 2]{ACS16}:
\begin{definition}
\label{D:ACSqty}
The \emph{$\ACS$ quantity} associated to the isometric immersion $M\subset \R^d$ at $p\in M$ is defined as
\begin{equation*} 
\ACS(X,N)=  \sum_{k=1}^{n-1}( \|\II(e_k,X) \|^2+  \|\II(e_k,N) \|^2 ) -\sum_{k=1}^{n-1}R^M(e_k,X,e_k,X)-\Ric^M(N,N)
\end{equation*}
where $X,N\in T_pM$ are such that $\|X\|=\|N\|=1$ and $\left<X,N\right>=0$; $R^M$ denotes the curvature tensor of $M$\footnote{We use the sign convention for $R$ such that $\sec(v\wedge w)= R(v,w,v,w)/\|v\wedge w\|^2$.}; and $e_1,\ldots e_{n-1}$ is an orthonormal basis of $N^\perp\subset T_pM$.
\end{definition}

The geometric significance of the $\ACS$ quantity stems from the following result (see \cite[Theorem A]{ACS16}):
\begin{theorem}[Ambrozio-Carlotto-Sharp]
\label{T:ACS}
Suppose $(M,g)$ admits an isometric immersion into a Euclidean space $\R^d$ such that, for all $p \in M$, and all $X,N\in T_pM$ with $\|X\|=\|N\|=1$ and $\left<X,N\right>=0$, one has $\ACS(X,N)<0$. Then every closed  immersed minimal hypersurface $\Sigma\subset M$ satisfies \[\Ind(\Sigma)\geq \binom{d}{2}^{-1}b_1(M) .\]
In particular, $(M,g)$ satisfies the conclusion of Conjecture  with $C=\binom{d}{2}^{-1}$.
\end{theorem}
The above result applies directly to most (and indirectly to all) compact rank one symmetric spaces (see \cite{ACS16}). 

It will be convenient to rewrite the $\ACS$ quantity in terms of $\II$ only. 
\begin{lemma}
\label{L:ACSqty}
In the notation of Definition \ref{D:ACSqty} above, 
\begin{align}
\label{E:ACS}
\ACS(X,N) &= -\left<H, \II(X,X) +\II(N,N)\right>
+2 \| \II(X,\cdot)\|^2 +2 \| \II(N,\cdot)\|^2 
 \\
 &\qquad +\left< \II(X,X), \II(N,N)\right>
-2\| \II(X,N)\|^2 - \| \II(N,N)\|^2 \nonumber
\end{align}
where $H$ denotes the mean curvature vector of $M\subset \R^d$; $\| \II(X,\cdot)\|  $ denotes the Frobenius norm
of the linear map $Y\mapsto \II(X,Y)$; and similarly for $\| \II(N,\cdot)\|  $.
\end{lemma}
\begin{proof}
First note that \[\sum_{k=1}^{n-1} \|\II(e_k,X) \|^2 = \| \II(X,\cdot)\|^2-\| \II(X,N)\|^2\]
 and similarly for $\sum_{k=1}^{n-1}\|\II(e_k,N) \|^2$. Next, use the Gauss equation to write
\begin{align*}
\Ric^M(N,N) &= \left<\II(N,N), H\right>-\| \II(N,\cdot)\|^2\\
\sum_{k=1}^{n-1}R^M(e_k,X,e_k,X) & =\Ric^M(X,X)-R^M(N,X,N,X)\\
& = \left<\II(X,X), H-\II(N,N)\right>-\| \II(X,\cdot)\|^2+\| \II(X,N)\|^2
\end{align*}
Putting these terms together yields the desired formula.
\end{proof}

\section{Robust index bound}
\label{S:robust}
The goal of this section is to prove Theorem \ref{MT:robust}. The main ingredient of our proof is also the main ingredient of the proof of the Nash embedding theorem. For convenience we will use the simplification of Nash's proof due to G\"unther \cite{Guenther89}.

Following Gromov-Rohlin \cite{GromovRohlin70}, we define the class of \emph{free} immersions:
\begin{definition}
A smooth immersion $u:M\to \R^d$ with second fundamental form $\II$ is called \emph{free} if, for any point $p\in M$, and any basis $\{e_1, \ldots e_n\}$ of $T_p M$, the normal vectors $\II(e_i,e_j)$ for $1\leq i\leq j\leq n$, are linearly independent.
\end{definition}
Note that if $N\to M$ is an immersion, and $M\to \R^d$ is a free immersion, then the composite immersion $N\to\R^d$ is free.

Fix a ``H\"older exponent''  $\lambda$ with $0<\lambda <1$, and denote by $\|\cdot\|_s$ the H\"older norm of a real-valued function on the open unit ball $B\subset \R^n$, given by
\begin{equation}
\label{E:Hoelder}
 \| u \|_s= \sum_{|\alpha|\leq s} \sup_{x\in U} |D^\alpha u(x)| +\sum_{|\alpha|= s} \sup_{x\neq y\in U} \frac{|D^\alpha u(x)-D^\alpha u(y)|}{|x-y|^\lambda}
\end{equation}
Fixing an atlas of $M$ and a partition of unity, there is an extension of the definition above to smooth functions on $M$, and sections of any vector bundle on $M$, all of which are still denoted by $\|\cdot\|_s$. We note that G\"unther uses a different definition (see \cite[page 70]{Guenther89}) of the H\"older norm, but it is equivalent to the more common definition  \eqref{E:Hoelder} above.

Let $u:M\to \R^d$ be a free immersion. G\"unther defines a map from the space of smooth symmetric $2$-tensors $C^\infty(M,\operatorname{Sym}^2T^*M)$ to the space of smooth normal sections, denoted by $f\mapsto E(u)(0,f)$, in the following way. For any $p\in M$, since $u$ is free, there exists a normal $v\in \nu_pM$ such that $\left<\II(X,Y),v\right>=f(X,Y)$ for all $X,Y\in T_pM$. Such $v$ is not unique, but selecting at every point the unique $v$ with minimal norm yields the normal vector field $E(u)(0,f)$. Moreover, there exist constants $K$, depending only on the fixed atlas and partition of unity, and $D(u)$, depending on these and the free immersion $u$, such that, for all  $f\in C^\infty(M,\operatorname{Sym}^2T^*M)$, the following inequality is satisfied (see \cite[Equation (34)]{Guenther89}):
\begin{equation}
\label{E:34}
\|E(u)(0,f)\|_2\leq K D(u)\|f\|_2.
\end{equation}

\begin{theorem}[\cite{Guenther89}]
\label{T:Guenther}
Let $M$ be a compact  manifold with fixed atlas and partition of unity as above. Then, there exists $\theta>0$ such that, for any free   immersion $u:M\to \R^d$, and $f\in C^\infty(M,\operatorname{Sym}^2T^*M)$ such that $D(u)\|E(u)(0,f)\|_2\leq \theta$, there exists $v\in C^\infty(M, \R^d)$ with $\|v\|_2\leq \|E(u)(0,f)\|_2$ such that $u+v$ is an isometric immersion with respect to the metric $g'=g+f$, where $g$ is the metric induced by $u$.
\end{theorem}

The following statement is an immediate consequence of Theorem \ref{T:Guenther} and \eqref{E:34}.
\begin{lemma}
\label{L:Nash}
Let $(M,g)$ be a compact Riemannian manifold, and $u:M\to \R^d$ a smooth free isometric immersion. Then, for any $\delta>0$, there is $\epsilon>0$ such that, for all $f\in C^\infty(M,\operatorname{Sym}^2T^*M)$ with $\|f\|_2<\epsilon$, there is $v\in C^\infty(M,\R^d)$ with $\|v\|_2<\delta$ such that $u+v$ is an isometric immersion with respect to the metric $g'=g+f$. 
\end{lemma}
\begin{proof}
Let $\delta>0$. Let $\theta>0$ satisfying the conclusion of Theorem \ref{T:Guenther}. Take
\[\epsilon=\min\left\{\frac{\delta}{KD(u)}, \frac{\theta}{K D(u)^2}\right\}.\]
Then, for any $f\in C^\infty(M,\operatorname{Sym}^2T^*M)$ with $\|f\|_2<\epsilon$,  \eqref{E:34} implies that 
\[D(u)\|E(u)(0,f)\|_2<\theta.\]
Thus, by Theorem \ref{T:Guenther}, there exists $v\in C^\infty(M,\R^d)$ with $\|v\|_2<\|E(u)(0,f)\|_2$ such that the immersion $u+v$ induces the metric $g+f$. By \eqref{E:34}, $\|v\|_2 <\delta$.
\end{proof}

\begin{proof}[Proof of Theorem \ref{MT:robust}]
We may assume, without loss of generality, that the image of the isometric immersion $u:(M,g)\to\R^d$ is contained in the unit sphere $S^{d-1}$ centered at the origin.

Consider for $\theta>0$ the ``Veronese'' embedding 
\[V_\theta:\R^d\to\R^d\times\R^{\binom{d+1}{2}}\]
\[V_\theta(y_1,\ldots y_d)=(y_1, \ldots, y_d, \theta y_1^2,\theta y_1y_2, \ldots, \theta y_d^2). \]
The orthogonal group $\OO(d)$ acts in the natural way, as linear isometries, on $\R^d$, hence on $\R^{\binom{d+1}{2}}=\operatorname{Sym}^2\R^d$, and diagonally on $\R^d\times\R^{\binom{d+1}{2}}$. With respect to these actions, $V_\theta$ is an $\OO(d)$-equivariant free immersion.

In particular, the metric on $S^{d-1}$ induced by the embedding $V_\theta|_{S^{d-1}}:S^{d-1}\to  \R^d\times\R^{\binom{d+1}{2}}$ is $\OO(d)$-invariant,  hence a constant scalar multiple of the original (round) metric, because the transitive action of $\OO(d)$ on $S^{d-1}$ is isotropy irreducible.

Therefore, the composition $V_\theta\circ u :M\to \R^d\times\R^{\binom{d+1}{2}}$ is free, and induces a constant scalar multiple $cg$ of the metric $g$. Moreover, if $\theta$ is small enough, $V_\theta\circ u$ has negative $\ACS$ quantity.

The maximum of the $\ACS$ quantity depends continuously on the immersion, with respect to the H\"older norm $\|\cdot\|_2$ (cf. \eqref{E:ACS}). Thus there is $\delta>0$ such that, if $v:M\to \R^d\times\R^{\binom{d+1}{2}}$ satisfies $\|v\|_2<\delta$, then the immersion $V_\theta\circ u+v$ has negative $\ACS$ quantity. By Lemma \ref{L:Nash}, there is $\epsilon>0$ such that, for all metrics $g'$ on $M$ with $\|g-g'\|_2<\epsilon$, there is   $v:M\to \R^d\times\R^{\binom{d+1}{2}}$ such that $V_\theta\circ u+v$ induces $cg'$ and has negative $\ACS$ quantity. By Theorem \ref{T:ACS}, if $g'$ is such a metric, and $\Sigma $ is a closed immersed minimal hypersurface in $(M,g')$, then
\[
\Ind(\Sigma) \geq  \binom{d+\binom{d+1}{2}}{2}^{-1} b_1(\Sigma)=\frac{8}{d(d+3)(d^2+3d-2)} b_1(\Sigma).
\]
\end{proof}

\section{Isoparametric examples}
\label{S:isoparametric}

\subsection{Preliminaries}
\label{SS:preliminaries}
We start by recalling some basic definitions and facts, as well as fixing the notation. We refer the reader to section 2.9 and chapter 4 of \cite{BCO} for a complete treatment.

A submanifold $M$ of Euclidean space (or the sphere, or hyperbolic space) is called \emph{isoparametric} if it has flat normal bundle, and constant principal curvatures along any parallel normal field. Using the Ricci equation, this implies that the tangent bundle $TM$ decomposes as the orthogonal direct sum of  \emph{curvature distributions} $E_i$, for $i=1, \ldots g$, which are the common eigenspaces for the shape operators. The eigenvalues are encoded in a family of parallel sections of the normal bundle, called the \emph{curvature normals} $\xi_i$: For any normal vector $\xi$, the shape operator $A_\xi$ in the direction of $\xi$ has eigenvalues $\left<\xi,\xi_i\right>$ with eigenspaces $E_i$. The dimensions $m_i$ of the curvature distributions $E_i$ are called  \emph{multiplicities}.

Given a parallel normal field $\xi$, the set $M_\xi=\{p+\xi(p)\ | \ p\in M\}$ is a smooth manifold. If $\dim(M_\xi)=\dim(M)$, $M_\xi$ is again an isoparametric submanifold, and is called a \emph{parallel} manifold to $M$. If $\dim(M_\xi)<\dim(M)$, $M_\xi$ is called a \emph{focal} manifold to $M$. It still has constant principal curvatures along any parallel normal field, but the normal bundle is no longer flat. The set of all parallel and focal manifolds of $M$ forms a singular Riemannian foliation of Euclidean space, called an \emph{isoparametric foliation}.


We will consider the case where $M^n\subset S^{n+1}\subset \R^{n+2}$. M\"unzner has shown that, in this case, the possible values for the number $g$ of principal curvatures are $1,2,3,4,6$. We will consider the case $g=4$ only, as it contains the richest class of examples.

To give a more explicit description of the parallel and focal manifolds, along with their second fundamental forms, we use the Coxeter group $W$ associated to $M$. Let $V$ be the normal space of $M$ at $p\in M$. $V$ is a two-dimensional subspace of $\R^{n+2}$ with $p\in V$, and is sometimes called a \emph{section} of the isoparametric foliation. There is a natural action of the dihedral group $W$ with $2g=8$ elements on $V$. The corresponding  reflection lines will be denoted $L_1, L_2, L_3, L_4$. Each parallel and focal manifold intersects $V$ in the orbit of a point under $W$. The focal submanifolds correspond to points in $L_1 \cup L_2\cup L_3\cup L_4$.

Choose an orthonormal basis of $V$ so that $L_i$ is the line orthogonal to the vector $\alpha_i$ for all $i$, where 
\[  \alpha_1=(1,-1) \quad \alpha_2=(1,0) \quad \alpha_3=(1,1) \quad \alpha_4=(0,1) .\]
The multiplicities satisfy $m_1=m_3$ and $m_2=m_4$, because $W$ acts on $\{L_1, L_2, L_3, L_4\}$ with orbits $\{L_1, L_3\}$ and $\{L_2, L_4\}$. The curvature normals at $p$ of the isoparametric submanifold  $M$ are given by 
\begin{equation}
\label{E:curvnormal}
\xi_i=-\frac{\alpha_i}{\left<\alpha_i,p\right>}.
\end{equation}
Note that $\left<\xi_i,p\right> =-1$ for every $i$. The second fundamental form satisfies
\begin{equation}
\label{E:2ndFF}
 \II(x_i, y_j)=\left<x_i,y_j\right>\xi_i 
\end{equation}
for $x_i\in E_i$ and $y_j\in E_j$. In particular, the mean curvature vector is given by $H=\sum_i m_i\xi_i$.

We will be interested in minimal (in the sphere) isoparametric submanifolds:
\begin{lemma}
\label{L:maxvol}
In the notation above, let $M$ be the isoparametric submanifold through $p=(\cos(\theta),\sin(\theta))\in V$ where $0<\theta<\pi/4$. The following are equivalent:
\begin{enumerate}[a)]
\item
$M$ is minimal in $S^{n+1}$.
\item 
$H=-np$.
\item
The volume of $M$ is maximal among its parallel hypersurfaces in the sphere.
\item $\theta=(1/2)\arctan(\sqrt{m_2/m_1})$.
\end{enumerate}
\end{lemma}
\begin{proof}
The mean curvature vector of $M$ in $\R^{n+2}$ is $H=\sum_i m_i\xi_i= -np+ H^S$, where $H^S$ denotes the mean curvature vector of $M$ in the sphere. Thus, (a) and (b) are equivalent.

We claim that, up to a  constant, the function $f(p)=\operatorname{Vol}(M)$ is given by
\[ f(p)= \left<\alpha_1,p\right>^{m_1}\left<\alpha_2,p\right>^{m_2}\left<\alpha_3,p\right>^{m_1}\left<\alpha_4,p\right>^{m_2}\]
Indeed, given two parallel isoparametric submanifolds $M,M'$, through points $p, p'\in V$, let $\eta=p'-p$, and extend it to a parallel normal field to $M$, also called $\eta$. Then the endpoint map $\phi:M\to M'$ given by $x\mapsto x+\eta(x)$ is a diffeomorphism. Its differential is block diagonal with $d\phi|_{E_i}=(1-\left< \eta, \xi_i\right>)\operatorname{Id}$. In particular,
\begin{align*}
\frac{\operatorname{Vol}(M')}{\operatorname{Vol}(M)}& =\pm  \det d\phi = \pm \prod_i(1-\left< \eta, \xi_i\right>)^{m_i}
= \pm \prod_i(1+\left< p, \xi_i\right> - \left< p', \xi_i\right>)^{m_i}\\
&=\pm \frac{\prod_i  \left<p',\alpha_i\right>^{m_i}}{\prod_i \left<p,\alpha_i\right>^{m_i}} 
\end{align*}
thus finishing the proof of the claim.

Plugging in $p=(\cos(\theta),\sin(\theta))$ yields
\begin{align*}
f(p)&=(\cos(\theta)-\sin(\theta))^{m_1}\cos^{m_2}(\theta)(\cos(\theta)+\sin(\theta))^{m_1}\sin^{m_2}(\theta)\\
&= \cos^{m_1}(2\theta)\sin^{m_2}(2\theta)/2^{m_2}.
\end{align*}
Differentiating with respect to $\theta$ and setting equal to zero gives us
\[ \cos^{m_1-1}(2\theta)\sin^{m_2-1}(2\theta)\big( -m_1\sin^2(2\theta)+m_2\cos^2(2\theta) \big)=0 .\]
So the unique critical point in the open interval $(0,\pi/4)$ is $(1/2)\arctan(\sqrt{m_2/m_1})$. It is the maximum because the volume goes to zero as $\theta\to 0$ or $\pi/4$, thus proving the equivalence of (c) and (d).

Finally, to prove that (b) and (c) are equivalent, note that
\[ \nabla (\log f)= \sum_i m_i \left<\alpha_i,p\right>^{-1}\alpha_i =-H \]
so that $f$ is maximum subject to $\|p\|^2=1$ if and only if $H$ is parallel to $p$, that is, if and only if $H=-np$. 
\end{proof}

\subsection{The $\ACS$ quantity of isoparametric submanifolds and their focal manifolds}

Since the second fundamental form of an isoparametric submanifold is easy to write down, we get a very explicit formula for the $\ACS$ quantity.  The exact second fundamental form of a focal manifold is more subtle, but we find estimates that suffice for our purposes. We state these formulas and estimates only in the special situation we are considering, namely minimal isoparametric hypersurfaces of the sphere, with $4$ principal curvatures, but  similar formulas hold for isoparametric submanifolds of general codimension, number of principal curvatures, and multiplicities.

\begin{lemma}
\label{L:isoparametricACS}
Assume $M\subset S^{n+1}$ is a minimal isoparametric hypersurface with four principal curvatures, and $p\in M$. 
Let $X,N\in T_pM$ with $\|X\|=\|N\|=1$ and $\left<X,N\right>=0$. Write $X=\sum_i x_i$ and $N=\sum_i y_i$, where $x_i,y_i\in E_i$. Then
\begin{align}
\ACS(X,N) &=-2n+
\sum_{i,j=1}^4\bigg((\|x_i\|^2-\|y_i\|^2)\|y_j\|^2 -2 \left<x_i,y_i\right> \left<x_j,y_j\right>\bigg)\left<\xi_i,\xi_j\right> \\\
&\quad +2\sum_{i=1}^4 \bigg(\|x_i\|^2+\|y_i\|^2\bigg) \|\xi_i\|^2. \nonumber
\end{align}
\end{lemma}
\begin{proof}
We use the expression for the $\ACS$ quantity given in Lemma \ref{L:ACSqty}, together with the formula for the second fundamental form in \eqref{E:2ndFF}. Since $M$ is minimal in the sphere, the first term in Lemma \ref{L:ACSqty}, namely $-\left<H, \II(X,X) +\II(N,N)\right>$, equals $-2n$.

We claim that the next term $2 \| \II(X,\cdot)\|^2$ equals $2\sum_i \|x_i\|^2 \|\xi_i\|^2$, and similarly for  $2 \| \II(N,\cdot)\|^2$. Indeed, take an orthonormal basis $e_1, \ldots e_n$ of $T_pM$ where each $e_k$ belongs to  $E_{i(k)}$ for some (unique) index $i(k)$. Then, 
\begin{align*}
 \| \II(X,\cdot)\|^2 &= \sum_k \|\II(X,e_k)\|^2= \sum_k \left< x_{i(k)},e_k\right>^2 \|\xi_{i(k)}\|^2 \\
 & =  \sum_i \sum_{k, i(k)=i} \left<x_i,e_k\right>^2 \|\xi_i\|^2 =\sum_i \|x_i\|^2 \|\xi_i\|^2.
\end{align*}

Since $\II(X,N)=\sum_i\left<x_i,y_i\right>\xi_i$, it follows that
\[-2\| \II(X,N)\|^2=-2\sum_{i,j}\left<x_i,y_i\right> \left<x_j,y_j\right>\left<\xi_i,\xi_j\right>,\]
and similarly for the two remaining terms $\left< \II(X,X), \II(N,N)\right>$ and $- \| \II(N,N)\|^2$. Adding everything we obtain the expression in the statement of the lemma.
\end{proof}

For simplicity we make a change of variables:
\begin{lemma}
\label{L:ACS'}
Let $M\subset S^{n+1}$ be a minimal isoparametric hypersurface with four principal curvatures, all of whose multiplicities are larger than one, and $p\in M$. Then the maximum of $\ACS(X,N)$ over all pairs $X,N\in T_pM$ with $\|X\|=\|N\|=1$ and $\left<X,N\right>=0$ is equal to the maximum of $\ACS'(s,t)$ over all $(s,t)\in \Delta^3\times \Delta^3$, where $\Delta^3$ is the standard $3$-simplex 
\[\Delta^3=\{u\in\R^4\ |\ u_i\geq 0 \ \forall i,\ u_1+u_2+u_3+u_4=1\}\]
and $\ACS'(s,t)$ is defined by
\begin{equation}
\label{E:ACS'}
\ACS'(s,t) =-2n+ 
\sum_{i,j=1}^4(s_i-t_i)t_j \left<\xi_i,\xi_j\right> 
 +2\sum_{i=1}^4 (s_i+t_i) \|\xi_i\|^2 .
\end{equation}
\end{lemma}
\begin{proof}
Given $X,N\in T_pM$ with $\|X\|=\|N\|=1$ and $\left<X,N\right>=0$, write $X=\sum_i x_i$ and $N=\sum_i y_i$, where $x_i,y_i\in E_i$. Define $s_i=\|x_i\|^2$,  $t_i=\|y_i\|^2$,  $s=(s_1, s_2, s_3, s_4)$, and $t=(t_1, t_2, t_3, t_4)$. Note that  $(s,t)\in \Delta^3\times \Delta^3$. Moreover, by Lemma \ref{L:isoparametricACS},
\[ \ACS(X,N)=\ACS'(s,t)  -2\| \II(X,N)\|^2 \leq \ACS'(s,t).\]
This shows that $\max \ACS\leq \max \ACS'$.

To prove the reverse inequality, let $(s,t)\in \Delta^3\times \Delta^3$. Then, since  $\dim(E_i)>1$ for all $i$, there exists  $(X,N)$ such that $s_i=\|x_i\|^2$,  $t_i=\|y_i\|^2$, and  $\left<x_i,y_i\right>=0$, for every $i$. In particular, $\left<X,N\right>=0$, and
\[ \| \II(X,N)\|^2= \sum_{i,j}\left<x_i,y_i\right> \left< x_j,y_j\right>\left<\xi_i,\xi_j\right>=0.\]
Thus $\ACS'(s,t)=\ACS(X,N)$, completing the proof.
\end{proof}

\begin{remark}
\label{R:numerics}
$\ACS'(s,t)$ is linear in $s$, so that the maximum over all $s$, given a fixed $t$, must occur for $s$ at one of the four vertices of $\Delta^3$. On the other hand, if one fixes $s$, then $\ACS'(s,t)$ is quadratic and \emph{concave} in $t$. This means that given explicit values of the curvature normals $\xi_i$, the maximum of the ACS quantity may be efficiently computed by solving four convex quadratic optimization problems over the $3$-simplex. In practice this can be done with interior-point methods, implemented in e.g. the CVXOPT package for Python.
\end{remark}

We start with a simple upper bound for the $\ACS$ quantity, which depends only on the multiplicities:
\begin{lemma}
\label{L:simplebound}
In the notations and assumptions of Lemma \ref{L:ACS'}, 
\[ \ACS\leq -2n+\frac{10(m_1+m_2)}{m_1}\left(1+\sqrt{\frac{m_2}{m_1+m_2}}\right).\]
\end{lemma}
\begin{proof}
First we claim that $\ACS'\leq -2n+5\max_k\|\xi_k\|^2$. Indeed,
\begin{align}
\ACS' &\leq -2n+ 
\sum_{i,j=1}^4s_it_j \left<\xi_i,\xi_j\right> 
 +2\sum_{i=1}^4 (s_i+t_i) \|\xi_i\|^2 \nonumber \\
&\leq  -2n+ \left(\sum_{i,j=1}^4s_i t_j  
 +2\sum_{i=1}^4 (s_i+t_i) \right) \max_k\|\xi_k\|^2 = -2n+5\max_k\|\xi_k\|^2. \nonumber
\end{align}
Since we are assuming $m_2\geq m_1$, we have $\theta\geq \pi/8$, so that $\xi_1$ is the curvature normal with the largest norm. Computing directly from \eqref{E:curvnormal}, we obtain
\begin{align*}
 \|\xi_1\|^2 &=\frac{2}{(\cos(\theta)-\sin(\theta))^2}
=\frac{2}{1-\sin(2\theta)}=\frac{2}{1-\sqrt{\frac{m_2}{m_1+m_2}}} \\
&= \frac{2(m_1+m_2)}{m_1}\left(1+\sqrt{\frac{m_2}{m_1+m_2}}\right).
\end{align*}
Therefore $\ACS\leq\ACS'\leq -2n+\frac{10(m_1+m_2)}{m_1}\left(1+\sqrt{\frac{m_2}{m_1+m_2}}\right)$.
\end{proof}

\begin{proof}[Proof of Theorem \ref{MT:isoparametric} part (a)]

First assume $m_1\geq 5$. Then, by Lemma \ref{L:simplebound},
\begin{align*}
\ACS &\leq -4(m_1+m_2) +\frac{10(m_1+m_2)}{m_1}\left(1+\sqrt{\frac{m_2}{m_1+m_2}}\right)\\
& < 4(m_1+m_2) \left(   -1 + \frac{5}{m_1} \right) \leq 0.
\end{align*}

Now let $m_1=4$. By Lemma \ref{L:maxvol}, since $M$ is minimal, it contains  the point $p=(\cos(\theta),\sin(\theta))\in V$, where $\theta=(1/2)\arctan(\sqrt{m_2/4})$. As $m_2$ goes to infinity, the curvature normals $\xi_2, \xi_3, \xi_4$ converge, while the norm of $\xi_1$ goes to infinity. More precisely, by the proof of Lemma \ref{L:simplebound},
\[ \|\xi_1\|^2=\frac{4+m_2}{2}\left(1+\sqrt{\frac{m_2}{4+m_2}}\right). \]

It is enough to show that the maximum of the $\ACS'(s,t)$ quantity of $M$ over $(s,t)\in\Delta^3\times \Delta^3$ is negative when $m_2$ is large (see Lemma \ref{L:ACS'}). Moreover, since the maximum occurs for $s$ at one of the vertices of $\Delta^3$, we need to show that, for $m_2$ large and $i=1,2,3,4$, the maximum of $\ACS'(e_i,t)$  over $t\in\Delta^3$ is negative.

Since $n=2(4+m_2)$, one has
\begin{align*}
\ACS'(s,t) &= -4(4+m_2)+ [(s_1-t_1)t_1 +2(s_1+t_1)]\|\xi_1\|^2+O(\sqrt{m_2}) \\
 &= \frac{4+m_2}{2} \left( -8 +[(s_1-t_1)t_1 +2(s_1+t_1)]\left(1+\sqrt{\frac{m_2}{4+m_2}}\right) \right)+O(\sqrt{m_2})
\end{align*}
If $s=e_i$ for $i\neq 1$, that is, if $s_1=0$, then $(s_1-t_1)t_1 +2(s_1+t_1)=-t_1^2+2t_1\leq 1$. This implies that $\max_t \ACS'(e_i,t)\to-\infty$ as $m_2\to \infty$.

Assume now that $s=e_1$. We claim that, for large $m_2$, the maximum of $\ACS'(e_1,t)$ occurs at $t=e_1$. This will finish the proof, because $\ACS'(e_1,e_1)=-2n+4\|\xi_1\|^2<0$.

To prove the claim, it is enough to show that (when $m_2$ is large) the gradient of $\ACS'(e_1,t)$ at $t=e_1$ has negative inner product with the vectors $e_2-e_1$, $e_3-e_1$, and $e_4-e_1$. This is because  $\ACS'(e_1,t)$ is concave, the simplex $\Delta^3$ is convex, and its tangent cone at $t=e_1$ is the cone over the  convex hull of $e_2-e_1$, $e_3-e_1$, and $e_4-e_1$.
But it is clear from the formula for $\ACS'$ that  \[\nabla\ACS'(e_1,t)|_{t=e_1}= \|\xi_1\|^2 e_1 +O(\sqrt{m_2})\]
thus finishing the proof.
\end{proof}

\begin{remark}
Based on numerical evidence (cf. Remark \ref{R:numerics}), we believe that the conclusion of Theorem \ref{MT:isoparametric}(a) also holds without the hypothesis ``$m_2$ is large enough''.
\end{remark}

\begin{proof}[Proof of Theorem \ref{MT:isoparametric}(b)]
The focal submanifold $M_+$ contains the point $p=(1,1)/\sqrt{2}$, and the tangent space at $p$ is $T_pM_+=E_2\oplus E_3\oplus E_4$. Write the normal space as an orthogonal direct sum $\nu_pM_+=U\oplus V$, where $U=E_1$, and $V$ is the same section as before, described in subsection \ref{SS:preliminaries}. Denote by $\II^U$ and $\II^V$ the components of the second fundamental form in the directions of $U$ and $V$.

By the Tube Formula (\cite[Lemma 3.4.7]{BCO}, $\II^V$ is given by the same formula as in the isoparametric case, see equations \eqref{E:curvnormal} and \eqref{E:2ndFF}. More precisely, for $i,j\in\{2,3,4\}$ and $x_i\in E_i$ and $y_j\in E_j$, one has $\II^V(x_i,y_j)=\left<x_i, y_j\right> \xi_i$, where
\[ \xi_2=(-\sqrt{2},0), \qquad  \xi_3=-(1,1)/\sqrt{2}=-p, \qquad  \xi_4=(0,-\sqrt{2}) .\]

As for $\II^U$, one can say that, for every $v\in U$ with $\|v\|=1$, the shape operators of $M_+$ in the directions of $v$ and $v_0=(1,-1)/\sqrt{2}\in V$ are conjugate. Indeed, for small $\epsilon$, $p+\epsilon v$ belongs to an isoparametric manifold parallel to $M$, whose normal space $V'$ contains $v$. Then, on $V'$ one has the same picture as in $V$ , with $v$ playing the role of $v_0$, so the result follows from the Tube Formula again. Explicitly, these shape operators have eigenvalues $\left<v_0,\xi_i\right>$, that is, $-1,0,1$, with multiplicities $m_2, m_1, m_2$, respectively. In particular, $M_+$ is minimal in the sphere, so that the $\ACS$ quantity of $M_+$ satisfies
\[ \ACS(X,N)+2(m_1+2m_2) \leq  
2 \| \II(X,\cdot)\|^2 +2 \| \II(N,\cdot)\|^2 
+\left< \II(X,X), \II(N,N)\right>.\]

The right-hand side equals the sum of the analogous expressions with $\II$ replaced with $\II^U$ and $\II^V$, respectively. The latter is at most $5\max_{i=2,3,4}\|\xi_i\|^2=10$, by the same argument as in the proof of  Lemma \ref{L:simplebound}. Thus
\[ \ACS \leq  -2(m_1+2m_2)+10+
2 \| \II^U(X,\cdot)\|^2 +2 \| \II^U(N,\cdot)\|^2 
+\left< \II^U(X,X), \II^U(N,N)\right> .\]

Since  the shape operator in the direction of any $v\in U$ with $\|v\|=1$ has largest eigenvalue (in absolute value) equal to $1$, we have $\| \II^U(X,\cdot)\|^2\leq \dim U=m_1$, and analogously for the other terms, so that
\[ \ACS \leq  -2(m_1+2m_2)+10+ 5m_1 .\]
Therefore, the $\ACS$ quantity is negative provided that $m_2>(3m_1+10)/4$. 
\end{proof}

\subsection{Examples}
\label{SS:examples}

Isoparametric hypersurfaces in spheres have been almost completely classified through the work of several mathematicians (see \cite[Section 2.9.6]{BCO}). The homogeneous ones are precisely the principal orbits of isotropy representations of rank two symmetric spaces. All known inhomogeneous isoparametric hypersurfaces in spheres have $4$ principal curvatures. They were constructed in \cite{FerusKarcherMunzner81} using Clifford systems, and are usually called of \emph{FKM-type}. We will identify some of these isoparametric foliations whose multiplicities satisfy the conditions in Theorem \ref{MT:isoparametric}. 

Starting with the homogeneous examples, consider the isotropy representations of the Grassmannians of two-planes over the reals,  complex numbers, or quaternions. Explicitly, given $k\geq 3$, let $G=\SO(k)\SO(2)$ (respectively $S(\U(k)\U(2))$, $\Sp(k)\Sp(2)$) act on the space of $k\times 2$ matrices with coefficients in $\R$ (respectively $\C$, $\HH$),  by $(A,B)C=ACB^{-1}$. The principal $G$-orbits are isoparametric hypersurfaces with multiplicities $(m_1,m_2)=(1,k-2)$ (respectively $(2,2k-3)$, $(4,4k-5)$). The singular $G$-orbit with codimension $m_1+1$ in the sphere is the Stiefel variety of $2$-planes in $k$-space.

Applying Theorem \ref{MT:isoparametric}(a) yields:
\begin{corollary}
For large $k$, the unique principal orbit of the representation of $\Sp(k)\Sp(2)$ on $k\times 2$ matrices with coefficients in $\HH$ that is minimal in the sphere has negative $\ACS$ quantity.
\end{corollary}

Applying Theorem \ref{MT:isoparametric}(b):
\begin{corollary}
The Stiefel variety of $2$-frames in $\R^k$ (respectively $\C^k$, $\HH^k$), with metric induced from the embedding in the sphere as the orbit of the matrix 
\[
\begin{pmatrix}
1 & 0\\
0 & 1 \\
0 & 0 \\
\vdots & \vdots \\
0 & 0
\end{pmatrix}
\]
under the the $G$-action described above, satisfies $\ACS<0$ provided that $k\geq 6$ (respectively $k\geq 4$, $k\geq 3$).
\end{corollary}

We can also apply Theorem \ref{MT:isoparametric} to one isolated example, which has multiplicities $(6,9)$:
\begin{corollary}
The unique minimal (in  $S^{31}$) principal orbit of the isotropy representation of the symmetric space $E_6/\operatorname{Spin}(10)\U(1)$ has $\ACS<0$. The singular orbit with codimension $7$ in $S^{31}$ also has $\ACS<0$.
\end{corollary}

Now we turn to isoparametric hypersurfaces of FKM-type \cite{FerusKarcherMunzner81}. Recall that a Clifford system on $\R^{2l}$ is a set of $m+1$ symmetric $2l\times 2l$ matrices $C=(P_0, \ldots, P_m)$ such that $P_i^2=I$ for all $i$, and $P_i P_j= -P_j P_i$ for $i\neq j$. Then $l$ needs to be of the form $l=k\delta(m)$, where $\delta(m)$ is described in Table \ref{table} (see \cite[page 483]{FerusKarcherMunzner81}). Conversely, given $m$ and $k$, there do exist Clifford systems as above (see \cite{FerusKarcherMunzner81} for a discussion of different equivalence relations of Clifford systems, and classification results). An isoparametric foliation of $S^{2l-1}$ is defined by the level sets of the polynomial $H(x)$ on $\R^{2l}$ given by
\[ H(x)= \sum_{i=0}^{m} (x^TP_i x)^2 ,\]
where $x$ is regarded as a column vector, and $x^T$ denotes its transpose.
The regular leaves are isoparametric with multiplicities $(m,l-m-1)$. The level set $H^{-1}(0)$ is one of the singular leaves, and it has codimension $1+m$ in $S^{2l-1}$. It is a quadric, because it can also be described as $\{x\in S^{2l-1}\ |\  x^TP_i x=0\ \forall i\}$, and it is sometimes called a \emph{Clifford-Stiefel variety}. The other singular leaf is $H^{-1}(1)$, and it has codimension $l-m$. In all but finitely many cases
\footnote{More precisely, the values of $(m,k)$ such that $0<l-m-1<m$ are $(2,2),(4,2),(5,1),(6,1),(8,2),(9,1)$.}
, $m\leq l-m-1$, so that, in our notation, $m_1=m$, $m_2=l-m-1$, and $M_+=H^{-1}(0)$.

\begin{table}[h!]
\centering
\begin{tabular}{| r |c c c c c c c c c c c|}  
 \hline
 $m=$ & 1 & 2 & 3 & 4 & 5 & 6 & 7 & 8 & \ldots & $m'+8$ & \ldots\\ 
\hline
 $\delta(m)$= & 1 & 2 & 4 & 4 & 8 & 8 & 8 & 8 & \ldots & $16\delta(m')$ & \ldots \\
 \hline
\end{tabular}
\caption{}
\label{table}
\end{table}

From  Theorem \ref{MT:isoparametric}(a) we have:
\begin{corollary}
Let $C=(P_0, \ldots, P_m)$ be a Clifford system on  $\R^{2l}$, with $l=k\delta(m)$. Then the unique regular leaf with maximal volume satisfies $\ACS<0$ provided: $m, l-m-1\geq 5$; or $m=4$ and $k$ is large enough.
\end{corollary}

Applying Theorem \ref{MT:isoparametric}(b), we obtain:
\begin{corollary}
\label{C:Clifford-Stiefel}
Let $C=(P_0, \ldots, P_m)$ be a Clifford system on  $\R^{2l}$, with $l=k\delta(m)$. Assume $k> \frac{7m+14}{4\delta(m)}$. Then the Clifford-Stiefel variety $M_+$ satisfies $\ACS<0$. 
\end{corollary}
\begin{proof}
We claim that $m\leq l-m-1$. Indeed, assuming $m> l-m-1$, we  obtain $k< \frac{2m+1}{\delta(m)}$, which, together with $k> \frac{7m+14}{4\delta(m)}$,  implies $m>10$. But then, by Table \ref{table}, $ \frac{2m+1}{\delta(m)}$ is less than one, contradicting the fact that $k\geq 1$ and proving the claim.

Thus $m_1=m$ and $m_2=k\delta(m)-m-1$, so that $k> \frac{7m+14}{4\delta(m)}$ implies $m_2>(3m_1+10)/4$, and we may apply Theorem \ref{MT:isoparametric}(b).
\end{proof}
Note that  Corollary \ref{C:Clifford-Stiefel}  applies to all but finitely many FKM-type isoparametric foliations.

\subsection{Remarks about the geometry of the examples}
In this subsection we collect a few remarks about the curvature and homogeneity of the isoparametric examples described above.

We start by relating the Ricci curvature and the $\ACS$ quantity of general isoparametric submanifolds of Euclidean space.

\begin{lemma}
\label{L:Ricci}
Let $M$ be an isoparametric submanifold of Euclidean space, with curvature distributions $E_i$ and curvature normals $\xi_i$, for $i=1,\ldots g$. Then the Ricci tensor of $M$ has $E_i$ as eigenspaces, with respective eigenvalues $\left<\xi_i, H\right>-\|\xi_i\|^2$.
\end{lemma}
\begin{proof}
Given $X=\sum_i x_i$ with $x_i\in E_i$, it follows from the Gauss equation that
\begin{align*} 
\Ric (X) &= \sum_k R(X,e_k,X,e_k)=\sum_k \left<\II(X,X),\II(e_k,e_k)\right> - \sum_k \|\II(X,e_k)\|^2  \\
 &= \left<\II(X,X),H\right> - \| \II(X,\cdot)\|^2 =\sum_i \big(\left<\xi_i, H\right>-\|\xi_i\|^2\big) \|x_i\|^2 
\end{align*}
where we have used the identity $ \| \II(X,\cdot)\|^2=\sum_i \|x_i\|^2 \|\xi_i\|^2$, see the proof of Lemma \ref{L:isoparametricACS}.
\end{proof}

\begin{proposition}
\label{P:Ricci}
Let $M$ be an isoparametric submanifold of Euclidean space, and assume the multiplicities $m_i=\dim E_i$ are larger than one. Then $\ACS<0$ implies $\Ric>0$.
\end{proposition}
\begin{proof}

We compute the $\ACS$ quantity as in Lemma \ref{L:isoparametricACS}. Let $p\in M$, 
 $X,N\in T_pM$ with $\|X\|=\|N\|=1$ and $\left<X,N\right>=0$. Write $X=\sum_i x_i$ and $N=\sum_i y_i$, where $x_i,y_i\in E_i$. Then
\begin{align*}
 \ACS(X,N) & = \sum_{i,j=1}^g\bigg((\|x_i\|^2-\|y_i\|^2)\|y_j\|^2 -2 \left<x_i,y_i\right> \left<x_j,y_j\right>\bigg)\left<\xi_i,\xi_j\right> +\\
& \sum_{i,j=1}^g -m_i(\|x_j\|^2+\|y_j\|^2)\left<\xi_i,\xi_j\right>+2\sum_{i=1}^g \bigg(\|x_i\|^2+\|y_i\|^2\bigg) \|\xi_i\|^2.
\end{align*}

As in Lemma \ref{L:ACS'}, the assumption that the multiplicities are greater than one implies that the maximum of the $\ACS$ quantity is equal to the maximum of 
\[ \ACS'(s,t) =
\sum_{i,j=1}^g\big( -m_i(s_j+t_j)+(s_i-t_i)t_j \big)\left<\xi_i,\xi_j\right> 
 +2\sum_{i=1}^g (s_i+t_i) \|\xi_i\|^2  \]
for $(s,t)\in \Delta^{g-1}\times \Delta^{g-1}$, where $\Delta^{g-1}$ denotes the standard $(g-1)$-simplex.
 
Let $k=1, \ldots, g$. Setting $s_i=t_i=\delta_{ik}$ for all $i$ in the equation above, we get $\ACS'(s,t)=-2\left<H,\xi_k \right>+4\|\xi_k\|^2$, where $H=\sum_i m_i \xi_i$ is the mean curvature vector. This is negative by assumption, so that, in particular, $\left<H, \xi_k\right> -\|\xi_k\|^2>0$. By Lemma \ref{L:Ricci}, these are the eigenvalues of the Ricci tensor.
\end{proof}

Thus, isoparametric hypersurfaces $M^n\subset S^{n+1}$ satisfying the conditions of Theorem \ref{MT:isoparametric}(a), and in particular all isoparametric examples listed in Subsection \ref{SS:examples}, have positive Ricci curvature. Similarly, one has:
\begin{proposition}
\label{P:focalRic}
The focal manifolds $M_+$ satisfying the conditions of Theorem \ref{MT:isoparametric}(b) have positive Ricci curvature.
\end{proposition}
\begin{proof}
We will freely use the notations and facts established in the proof of Theorem \ref{MT:isoparametric}(b). Let $X=x_2+ x_3+ x_4\in T_p M_+-\{0\}$ with $x_i\in E_i$ for $i=2,3,4$. Then, by the Gauss equation, 
\begin{align*}
 \Ric(X,X) &=  \left< \II(X,X) , H\right> -\|\II(X,\cdot) \|^2 \\
 &= (m_1+ 2m_2) \| X \|^2 -\|\II^U(X,\cdot) \|^2 -\|\II^V(X,\cdot) \|^2
 \end{align*}
because $M_+$ is minimal in the sphere.

But $\|\II^V(X,\cdot) \|^2 = \sum \| \xi_i \|^2 \| x_i\|^2=2 \|x_2\|^2 +  \|x_3\|^2 + 2 \|x_4\|^2$, while $\|\II^U(X,\cdot) \|^2\leq m_1\|X\|^2$. Therefore, $\Ric(X,X)>0$.
\end{proof}

\begin{proposition}
\label{P:sec}
Let $M_i$ be a sequence of minimal isoparametric hypersurfaces of spheres with four principal curvatures, with $m_1$ fixed, and $m_2\to\infty$. Then the minimum of the sectional curvatures of $M_i$ diverges to $-\infty$, while the diameter is bounded from below by $\pi$.
\end{proposition}
\begin{proof}
The diameter of $M_i$ is at least $\pi$ because it is contained in the sphere, and is invariant under the antipodal map. By the Gauss equation, the sectional curvature of a plane of the form $x\wedge y$, for $x\in E_1$ and $y\in E_4$ is:
\[ \sec(x\wedge y)=R(x, y,x,y)= \left<\II(x,x), \II(y,y)\right> - \| \II(x,y) \|^2= \left< \xi_1, \xi_4 \right> <0 .\]
Moreover, as $m_2 \to \infty$, $\xi_4\to(0,-\sqrt{2})$, so that $\sec(x\wedge y)$ is asymptotic to
\[ -\| \xi_1\| = -\sqrt{\frac{2(m_1+m_2)}{m_1}\left(1+\sqrt{\frac{m_2}{m_1+m_2}}\right)} \simeq -C m_2^{1/2},\]
for a positive constant $C$ (see proof of Lemma \ref{L:simplebound}). In particular, the minimum value of the sectional curvature of $M$ diverges to $-\infty$ as $m_2\to\infty$.
\end{proof}

\begin{remark}
\label{R:homogeneous}
It has been determined in \cite{FerusKarcherMunzner81} exactly which isoparametric hypersurfaces of FKM-type are extrinsically homogeneous. In particular, when $m\leq l-m-1$, so that, in our notation, $m_1=m$ and $m_2=l-m-1$, they prove that $M$ is extrinsically homogeneous if and only if $m=1,2$; or $m=4$ and $P_0 P_1 P_2 P_3 P_4=\pm I$.

Moreover, for isoparametric hypersurfaces $M\subset S^{n+1}$ with four principal curvatures, extrinsic and intrinsic homogeneity are equivalent. Indeed, the rank of the shape operator in the sphere is constant and $\geq 2$. Thus, we may apply \cite[Theorem 2]{Ferus70} to conclude that the embedding is rigid, so that, in particular, every isometry of $M$ extends to an isometry of $S^{n+1}$.

On the other hand, any isoparametric submanifold $M$ is curvature-homogeneous, by the Gauss equation and the fact that the second fundamental form is ``the same'' everywhere. More precisely, given $p,q\in M$, any linear isometry $T_pM\to T_qM$ that sends each curvature distribution $E_i(p)$ to $E_i(q)$ maps the curvature operator at $p$ to the one at $q$.
\end{remark}

\section{Symmetric examples}
\label{S:symmetric}

\subsection{Embeddings of symmetric spaces}

The goal of this section is to prove Theorem \ref{MT:symmetric}, whose parts (a), (b), and (c)  correspond to Propositions \ref{P:quatGrass}, \ref{P:Sp(n)}, and \ref{P:SU(n)}, respectively.
First we recall some well-known facts about symmetric spaces and their equivariant embeddings into Euclidean spaces. References for this material are \cite{Wallach72b}, \cite[Chapter 7]{Besse}.

Let $K\subset G$ be compact Lie groups. Recall that $(G,K)$ is called a \emph{symmetric pair}, and $G/K$ a (compact) \emph{symmetric space}, if there is an order two automorphism $\tau:G\to G$ such that $(G^\tau)_0\subset K\subset G^\tau$. Here $G^\tau=\{ g\in G\ |\ \tau(g)=g \}$, and $(G^\tau)_0$ denotes the connected component of $G^\tau$.

For example, every compact Lie group $G$ is diffeomorphic to the symmetric space $(G\times G)/\Delta G$. Here the automorphism $\tau$ is given by $\tau(a,b)=(b,a)$, and $\Delta G=(G\times G)^\tau=\{(a,a)\ |\ a\in G\}$. The diffeomorphism is given by $a\in G\mapsto (a,e)K\in G\times G/\Delta G$, where $e\in G$ denotes the identity element.

Let $(G,K)$ be a symmetric pair with $G,K$ compact, and denote by $\kk\subset\g$ their Lie algebras. Moreover, let $\m\subset \g$ be the $\Ad_K$-invariant complement of $\kk$ in $\g$ given by the $(-1)$-eigenspace of the differential of $\tau$ at the identity. Then $[\kk,\kk]\subset\kk$, $[\kk,\m]\subset\m$, and $[\m,\m]\subset \kk$. The tangent space of $G/K$ at $eK$ can be identified with $\m$, and the $G$-invariant metrics on $G/K$ correspond to the $\Ad_K$-invariant inner products on $\m$. When $G$ is semi-simple, a natural choice for such a metric is $(-B)|_\m$, where $B:\g\times \g\to\R$ is the Cartan-Killing form, defined by $B(X,Y)=\tr(\ad_X\circ\ad_Y)$.

\begin{lemma}
\label{L:Einstein}
Let $(G,K)$ be a symmetric pair of compact Lie groups with $G$ semi-simple, and let $B$ denote the Cartan-Killing form on $\g$.
\begin{enumerate}[a)]
\item
$(\N,g)=(G,-B)$ is Einstein with $\Ric=(1/4)g$.
\item
$(\N,g)=(G/K,-B|_\m)$ is Einstein with $\Ric=(1/2)g$.
\end{enumerate}
\end{lemma}
\begin{proof}
For the symmetric space, see \cite[Theorem 7.73]{Besse}. For the Lie group, note that the diffeomorphism $G\to G\times G /\Delta G$ is an isometry with respect to the metrics $-B$ and $-2(B\oplus B)$. Since the Ricci tensor is scale-invariant, it follows that
\[ \Ric^{(G,-B)}=\Ric^{(G\times G/\Delta G,-2(B\oplus B)}=-(1/2)B\oplus B=-(1/4)B. \]
\end{proof}

Now we consider $G$-equivariant embeddings of $G/K$ into Euclidean space:
\begin{lemma}
\label{L:second}
Let $(G,K)$ be a symmetric pair of compact Lie groups. Let $\rho:G\to O(V)$ an orthogonal representation of $G$ on the Euclidean space $V$, and let $p\in V$ with isotropy $K=G_p$. Denote by $I\!I$ the second fundamental form of the embedding of $G/K$ as the $G$-orbit $G\cdot p\subset V$ given by $aK\in G/K \mapsto \rho(a)p\in V$. Then
\[ I\!I(d\rho(X),d\rho(Y))=d\rho(X)d\rho(Y)p\]
for all $X,Y\in\m$. 
\end{lemma}
\begin{proof}
Let $\xi\in T_p(G\cdot p)$ be a normal vector to the orbit $G\cdot p$ at $p\in V$, and $X,Y\in \m$. Extend $\xi$ to a vector field along the curve $t\mapsto \rho(e^{tX})p$ by the formula $ \hat{\xi}(t)=\rho(e^{tX})\xi$, so that
\[ \nabla^V_{d\rho(X)}\hat{\xi}=\left.\frac{d}{dt}\right|_{t=0}\rho(e^{tX})\xi=d\rho(X)\xi.\]
Then
\begin{align*}
\left< I\!I (X,Y) ,\xi \right> &=   \left< S_\xi(d\rho(X)) , d\rho (Y)p \right> &  \\
&= -\left< \nabla^V_{d\rho(X)}\hat{\xi}, d\rho (Y)p\right>  &\\
&= -\left< d\rho (X)\xi, d\rho (Y)p\right> & \\
&= \left< d\rho(X)d\rho(Y)p,\xi\right> &\text{ because }d\rho(X)^T=-d\rho(X) 
\end{align*}
It remains to show that $d\rho(X)d\rho(Y)p$ is normal to the orbit, or, equivalently, that the trilinear tensor $\eta:\m\times\m\times\m\to\R$ defined by \[\eta(X,Y,Z)=\left< (d\rho(X)d\rho(Y)p\ ,\  d\rho(Z)p\right> \] vanishes identically. 

Since $(G,K)$ is a symmetric pair, $[X,Y]\in\kk$, and in particular $d\rho([X,Y])p=0$. This means that
\begin{align*}
\eta(X,Y,Z) &=\left<d\rho(X)d\rho(Y)p,d\rho(Z)p\right> \\
&=-\left<d\rho(X)d\rho(Z)p,d\rho(Y)p\right>\\
&=-\eta(X,Z,Y)\\
&=-\eta(Z,X,Y)
\end{align*}
Since the permutation $(X,Y,Z)\mapsto (Z,X,Y)$ has order three, we conclude $\eta=-\eta$, that is, $\eta=0$.
\end{proof}

\subsection{Rewriting the ACS quantity}

We start with an equivalent reformulation of the ACS quantity.
\begin{lemma}
Denoting by $H$ the mean curvature vector of the embedding $M \subset\R^d$,
\begin{align}
\ACS &=-2\Ric(X,X)-2\Ric(N,N)+\left< H, I\!I(X,X)+ I\!I(N,N)\right> &  \label{E:ACS2} \\
& \quad -2\|I\!I(X,N)\|^2 -\|I\!I(N,N)\|^2+\left< I\!I(N,N),I\!I(X,X) \right> . \nonumber
\end{align}
\end{lemma}
\begin{proof}
Use Lemma \ref{L:ACSqty} and the formula
$\Ric(X,X)=\left< \II(X,X), H \right> -\| \II(X,\cdot) \|^2$,
which is a consequence of the Gauss equation.
\end{proof}

Now assume $(M,g)$ is Einstein with $\Ric=E.g$ and the embedding $M\subset \R^d$ is minimal into a sphere $S(r)\subset \R^d$. Then:
\begin{equation}
\ACS= -4E+\frac{2\dim(M)}{r^2}-2\|I\!I(X,N)\|^2 -\|I\!I(N,N)\|^2+\left< I\!I(N,N),I\!I(X,X) \right> .
\label{E:ACS3}
\end{equation}
We will refer to the term $-4E+2\dim(M)/r^2$ in \eqref{E:ACS3} as the \emph{constant term}. 

\begin{remark}
The coordinates of the embedding $M\subset \R^d$ are eigenfunctions of the Laplace-Beltrami operator with eigenvalue $\lambda=\dim(M)/r^2$ (see \cite[Cor. 5.2]{Wallach72b}). By Lichnerowicz's Theorem, $\lambda\geq E \dim(M)/(\dim(M)-1)$, so that the constant term satisfies 
\[-4E+\frac{2\dim(M)}{r^2}\geq \frac{-2E(\dim(M)-2)}{\dim(M)-1}.\]
\end{remark}

\subsection{Unitary groups}

Let $G=\SU(n),\Sp(n)$, and consider their natural embedding into $V=\C^{n\times n},\HH^{n\times n}$, as $n\times n$  complex-unitary and quaternionic-unitary matrices. We endow $G$ with the metric given by the negative of the Cartan-Killing form $B:\g\times \g\to \R$, and extend it to the inner product $\left<\cdot,\cdot\right>$ on $V$ defined by 
\begin{equation}\label{E:Killing}  
\left<X, Y\right> =c_n\Re(\tr(XY^*)) \end{equation}
where $c_n$ equals $2n$, $4(n+1)$ in the complex and quaternionic cases, respectively, and $\Re$ denotes the real part.
\begin{lemma}
\label{L:ACSclassicalgroups}
With the notations above, the $\ACS$ quantity of the isometric embedding $G\subset V$ is given by:
\[
\ACS=
\begin{cases}
-\frac{1}{n^2}-\left<NX,XN\right>-\| N^2 \|^2  &\text{in the complex case} \\
-\frac{1}{2(n+1)}-\left<NX,XN\right>-\| N^2 \|^2&\text{in the quaternionic case}
\end{cases}
\]
where $X,N\in\g$ such that $\|X\|=\|N\|=1$ and $\left<X,N\right>=0$.
\end{lemma}
\begin{proof}
The image of $G\subset V$ is contained in the sphere of radius $r=\sqrt{nc_n}$.

The group $G\times G$ acts orthogonally on $V$ through the representation $\rho(A,B)Z=AZB^{-1}$, whose derivative is given by
$d\rho(X,Y)Z=XZ-ZY$. The point $p=I\in V$ has isotropy $\Delta G$, and the embedding $G\subset V$ factors as $G=G\times\{e\}\to G\times G/\Delta G \to V$, with the last map given by $(A,B)\Delta G\mapsto \rho(A,B)p$. By Lemma \ref{L:second}, the second fundamental form is given by 
\[ I\!I(X,Y)=d\rho\left(\frac{(X,-X)}{2}\right)d\rho\left(\frac{(Y,-Y)}{2}\right)p= d\rho\left(\frac{(X,-X)}{2}\right)Y=\frac{XY+YX}{2}. \]
It follows from an easy computation that the embedding of $G$ in the sphere of radius $r=\sqrt{nc_n}$ is minimal.

The constant term in \eqref{E:ACS3} is:
\[
-4E+2\frac{\dim(G)}{r^2}=
\begin{cases}
-1+2\frac{n^2-1}{nc_n}=-\frac{1}{n^2} &\text{in the complex case} \\
-1+2\frac{(2n+1)n}{nc_n}=-\frac{1}{2(n+1)} &\text{in the quaternionic case}
\end{cases}
\]

Let $X,N\in\g$ be a pair of orthogonal unit vectors.  The non-constant term in \eqref{E:ACS3} is:
\begin{align*}
 &-2\|I\!I(X,N)\|^2 -\|I\!I(N,N)\|^2+\left< I\!I(N,N),I\!I(X,X) \right>=\\
 &\qquad=-2\| (XN+NX)/2 \|^2 -\| N^2 \|^2+\left<N^2,X^2\right> \\
 &\qquad=-\left<NX,XN\right>-\| N^2 \|^2.
\end{align*}
We have used the identities $\|NX\|^2=\|XN\|^2=\left< N^2, X^2\right>$, which follow from the definition of the inner product and the assumption that $X,N$ are skew-Hermitian. 
\end{proof}

\begin{proposition}
\label{P:Sp(n)}
The standard isometric embedding  $(\Sp(n),-B)$ into $\HH^{n\times n}$ satisfies $\ACS<0$. In particular,  every closed embedded minimal hypersurface $M\subset \Sp(n)$ satisfies
\[\Ind(M)\geq \binom{4n^2}{2}^{-1} b_1(M).\]
\end{proposition}
\begin{proof}
Note that the so-called Frobenius inner product, given by $\left<\cdot,\cdot\right>_F=(4n+4)^{-1}\left<\cdot,\cdot\right>$, is sub-multiplicative, which implies that 
\[\left<NX,XN\right>\geq -\|NX\|^2\geq -\frac{\|X\|^2 \|N\|^2}{4n+4}=-\frac{1}{4n+4}\]
Therefore, using  Lemma \ref{L:ACSclassicalgroups}, we conclude that $\ACS<0$:
\[-\frac{1}{2(n+1)}-\left<NX,XN\right>-\| N^2 \|^2\leq -\frac{1}{4n+4}<0.\]

The stated bound for the index of $M$ follows from Theorem \ref{T:ACS}.
\end{proof}

To treat the case $G=\SU(n)$ we need a lemma:
\begin{lemma}
\label{L:SU(n)}
Let $n\geq 2$. The function $\tr((XN)^2+N^4)$ is real-valued on the set $ \{(X,N)\in\mathfrak{su}(n)^2\ |\ \tr (X^2)=\tr (N^2)=-1,\ \tr(XN)=0 \}$. Let $a_n$ denote its minimum value in this set. Then:
\begin{enumerate}[a)]
\item If $n$ is even, then $a_n=\frac{2-n}{8n}$.
\item If $n$ is odd,
$ \frac{3-n}{8(n-1)}=a_{n-1}\geq a_n\geq a_{n+1}=\frac{1-n}{8(n+1)}$.
\end{enumerate}
\end{lemma}
\begin{proof}
The functions $\tr((XN)^2)$ and $\tr(N^4)$ are real-valued because $X,N$ are skew-Hermitian.
\begin{enumerate}[a)]
\item Since $\tr((XN)^2+N^4)$ is invariant under simultaneous conjugation of $X$ and $N$ by $\SU(n)$, we may assume that $N=i\operatorname{diag}(z_1,\ldots z_n)$, where $z_j\in\R$, $z_1+\ldots +z_n=0$, and $z_1^2+\ldots +z_n^2=1$.

Fixing such $N$, the (real-valued) function $X\mapsto \tr((XN)^2)$ is quadratic, hence achieves its minimum at an eigenvector associated to the smallest eigenvalue of the map $X\mapsto -NXN$. Thus 
\[ \min_X \tr((XN)^2+N^4)= \min_{i<j}z_iz_j+\sum_kz_k^4\]
Since the minimum of $z_iz_j+\sum_kz_k^4 $  does not depend on $i,j$, it suffices to show that one of them, say $z_1z_2+\sum_k z_k^4$, has minimum $\frac{2-n}{8n}$ with the constraints that $z_1+\ldots +z_n=0$ and $z_1^2+\ldots + z_n^2=1$.

Moreover, since $n$ is even, it is enough to prove the Claim below. Indeed, the minimum of $z_1z_2+\sum_kz_k^4 $ will then be achieved at 
\[(z_1,z_2,\ldots, z_n)=\left(-\sqrt{\frac{n+2}{4n}}, +\sqrt{\frac{n+2}{4n}}, -\sqrt{\frac{1}{2n}},+ \sqrt{\frac{1}{2n}}, \ldots -\sqrt{\frac{1}{2n}},+ \sqrt{\frac{1}{2n}}\right),\]
because, by the Claim, this is the point where the minimum of $z_1z_2+\sum_kz_k^4 $ subject only to $z_1^2+\ldots + z_n^2=1$ is achieved, and this point happens to also satisfy the other constraint  $z_1+\ldots +z_n=0$.

\textbf{Claim:} The minimum of $f=-\sqrt{w_1w_2}+\sum_jw_j^2$ subject to $\sum_j w_j=1$ and $w_j\geq 0 \ \forall j$ equals   $\frac{2-n}{8n}$ and is achieved at 
\[w_1=w_2=\frac{n+2}{4n}\qquad w_j=\frac{1}{2n},\ j=3,\ldots n.\]

We prove the Claim by induction on $n$. The base case $n=2$ is straightforward. Assume $n>2$. We use Lagrange multipliers:
\[ \nabla f=\left(2w_1-\frac{w_2}{2\sqrt{w_1w_2}}, 2w_2-\frac{w_1}{2\sqrt{w_1w_2}}, 2w_3,2w_4,\ldots, 2w_{n}\right)\]
The equation $\nabla f=a(1,\ldots 1)$ for some $a\in\R$ implies that
\[w_1-w_2=-4(w_1-w_2)\sqrt{w_1 w_2}\]
and therefore $w_1=w_2=\frac{2a+1}{4}$ and $w_j=\frac{a}{2}$ for $j=3,\ldots n$. Since $\sum_j w_j=1$, we have exactly one critical point in the interior of the region defined by $w_j\geq 0$ for all $j$, namely
\[w_c=\left(\frac{n+2}{4n},\frac{n+2}{4n},\frac{1}{2n}, \ldots, \frac{1}{2n}\right)\]
Note that $f(w_c)=\frac{2-n}{8n}$.

On the other hand, assume $w=(w_1,\ldots w_n)$ lies on the boundary, that is, $w_j=0$ for some $j$. If $j=1,2$, then $f(w) \geq 0>\frac{2-n}{8n}$. If $j>2$, then by the inductive hypothesis we have $f(w)\geq \frac{2-(n-1)}{8(n-1)}>\frac{2-n}{8n}$. This concludes the proof of the Claim.

\item It is true for all $n\geq 2$ that $a_n\geq a_{n+1}$. Indeed, the sets \[ S_n= \{(X,N)\in\mathfrak{su}(n)^2\ |\ \tr (X^2)=\tr (N^2)=-1,\ \tr(XN)=0 \}\] satisfy $S_n\subset S_{n+1}$, and the function $\tr((XN)^2+N^4)$ on $S_n$ is the restriction to $S_n$ of the corresponding function on $S_{n+1}$. The stated result then follows from  (a). 
\end{enumerate}
\end{proof}

\begin{proposition} Consider the standard isometric embedding of $(\SU(n),-B)$ into $\C^{n\times n}$, where $B$ denotes the Cartan-Killing form.
\label{P:SU(n)}
\begin{enumerate}[a)]
\item Suppose $n<18$. Then the embedding  satisfies $\ACS<0$. In particular,  every closed embedded minimal hypersurface $M\subset \SU(n)$ satisfies
\[\Ind(M)\geq \binom{2n^2}{2}^{-1} b_1(M).\]
\item If $n>18$, the embedding $\SU(n)\subset \C^{n\times n}$ does not satisfy $\ACS<0$.
\item The embedding  $\SU(18)\subset \C^{18\times 18}$ satisfies $\ACS\leq 0$.
\end{enumerate}
\end{proposition}
\begin{proof}
By Lemma \ref{L:ACSclassicalgroups}, it is enough to determine the sign of 
\[b_n=\min\left(\frac{1}{n^2}+\left<NX,XN\right>+\| N^2 \|^2\right)\]
where the minimum is taken over all $X,N\in\mathfrak{su}(n)$ such that $\|X\|=\|N\|=1$ and $\left<X,N\right>=0$.

We claim that $b_n=\frac{1}{n^2}+\frac{a_n}{2n}$, where $a_n$ is  defined in Lemma \ref{L:SU(n)}. Indeed, letting $X'=\sqrt{2n} X$ and $N'=\sqrt{2n} N$, it follows that $\|X\|=1$ if and only if $\tr((X')^2)=-1$, and similarly for $N, N'$. Thus
\begin{align*}
b_n &=\frac{1}{n^2}+2n\min\left(\tr(NXNX)+\tr(N^4)\right)\\
 &=\frac{1}{n^2}+\frac{1}{2n}\min\left(\tr(N'X'N'X')+\tr((N')^4)\right)=\frac{1}{n^2}+\frac{a_n}{2n}
\end{align*}
where the first minimum is taken over $X,N\in\mathfrak{su}(n)$ such that $\|X\|=\|N\|=1$ and $\left<X,N\right>=0$, while the second minimum is taken over $X',N'\in\mathfrak{su}(n)$ such that $\tr ((X')^2)=\tr ((N')^2)=-1$, and $ \tr(X'N')=0$. This finishes the proof of the claim.

If $n$ is even, then by  Lemma \ref{L:SU(n)}, $a_n=\frac{2-n}{8n}$, so that $b_n=\frac{18-n}{16n^2}$. Therefore (c) and the statements in (a), (b) with $n$ even follow.

If $n$ is odd, then Lemma \ref{L:SU(n)} implies that
\[ \frac{1}{n^2}+\frac{a_{n+1}}{2n}=\frac{-n^2+17n+16}{16n^2(n+1)}
\leq b_n \leq
 \frac{-n^2+19n-16}{16n^2(n-1)}=\frac{1}{n^2}+\frac{a_{n-1}}{2n} \]
 In particular, $n<18$ implies $b_n>0$ and $n>18$ implies $b_n<0$, proving (a), (b) for $n$ odd. 
\end{proof}

\subsection{Quaternionic Grassmannians}
Given $d\leq n$, consider the Grassmannian manifold of $d$-planes in  $\HH^n$. It is a symmetric space which we will write as
 \[G/K= \frac{\Sp(n)}{\Sp(d)\times \Sp(n-d)}. \]
We endow $G/K$ with the metric induced from the Killing form on $G$ (see \eqref{E:Killing}), so that $G/K$ is Einstein with constant $1/2$, see Lemma \ref{L:Einstein}. The $\Ad_K$-invariant complement $\m$ of $\kk$ in $\g$ consists of the matrices 
\[ \hat{X}=\begin{bmatrix}
0 & X & \\
-X^* & 0
\end{bmatrix}  \]
where $X$ is a $d\times (n-d)$ matrix with entries in $\HH$.

Let $V$ be the space of traceless Hermitian $n\times n$ matrices, and endow $V$ with the ``same'' metric as $\g$, given by  \eqref{E:Killing}. The group $G$ acts on $V$ by conjugation, and the orbit through $p\in V$ is an isometric embedding of $G/K$ into $V$, with the metrics defined above, where
\[ p=\frac{1}{n}\begin{bmatrix}
(n-d)I_d &  0 \\
0 & -dI_{n-d}
\end{bmatrix}  \]
(and $I_k$ denotes the $k\times k$ identity matrix).

\begin{lemma}
\label{L:ACSgrassmannians}
With the notations above, the $\ACS$ quantity of the isometric embedding $G/K\subset V$ is given by:
\[
\ACS= -\frac{2}{n+1}-8c_n \Re \tr(X N^* X N^*+NN^* N N^*)
\]
where $c_n=4(n+1)$, and $X,N\in \HH^{d\times (n-d)}$, such that 
\[\tr(XN^*)=0, \quad \tr(XX^*)=\tr(NN^*)=\frac{1}{2c_n}.\]
\end{lemma}
\begin{proof}
The image of $G/K\subset V$ is contained in the sphere of radius $r$, where
 \[r^2=c_n\frac{d(n-d)}{n}.\]
 
By Lemma \ref{L:second}, the second fundamental form is given by 
\[ I\!I(X,Y)=d\rho(\hat{X})d\rho(\hat{N})p= -\begin{bmatrix}
X N^*+N X^* &  0 \\
0 & -(X^*N + N^*X)
\end{bmatrix}   \]
From this, an easy computation shows that the embedding $G/K\subset V$ is minimal. The constant term in  \eqref{E:ACS3} is:
\[
-4E+2\frac{\dim(G/K)}{r^2} =-2+2\frac{4d(n-d)}{c_nd(n-d)/n}=-\frac{2}{(n+1)} 
\]

Let $X,N\in\m$ be a pair of orthogonal unit vectors. A straight-forward computation yields the non-constant term in \eqref{E:ACS3}:
\begin{align*}
 &-2\|I\!I(X,N)\|^2 -\|I\!I(N,N)\|^2+\left< I\!I(N,N),I\!I(X,X) \right>=\\
 &\qquad= -c_n\Re\tr\begin{bmatrix}
 2(X N^*+N X^*)^2   &  0 \\
 0   &  2(X^* N+N^* X)^2 
 \end{bmatrix}  \\
 &\qquad\qquad- c_n\Re\tr\begin{bmatrix}
 4(N N^*)^2   &  0 \\
 0   &   4(N^* N)^2 
 \end{bmatrix}  \\
 &\qquad\qquad+ c_n\Re\tr\begin{bmatrix}
4N N^* X X^*   &  0 \\
 0   &  4N^* N X^* X
 \end{bmatrix}  \\
 &\qquad=-8c_n\Re\tr (XN^*XN^*+NN^*NN^*)
\end{align*}
Adding the constant and non-constant terms we arrive at the stated formula for $\ACS$.
\end{proof}

\begin{proposition} 
\label{P:quatGrass}
Let $d\leq n$, and let $(M,g)$ be the quaternionic Grassmannian of $d$-planes in $n$-space. The standard embedding of $(M,g)$ into the space of traceless Hermitian $n\times n$ matrices satisfies $\ACS<0$ for every $d,n$.
\end{proposition}
\begin{proof}
We use the formula for $\ACS$ stated in Lemma \ref{L:ACSgrassmannians}.
Recall that the ``Frobenius'' norm $\|A\|_F^2=\Re\tr(AA^*)$ on matrices is submultiplicative. Thus 
\[ -\frac{2}{n+1}-8c_n \Re \tr(X N^* X N^*+NN^* N N^*) \leq 
-\frac{2}{n+1} +8c_n \frac{1}{4c_n^2} =-\frac{3}{2(n+1)} <0\]
because $c_n=4(n+1)$ in the quaternionic case. 
\end{proof}

\begin{remark}
The natural embeddings of the group $\SO(n)$, and the real and complex Grassmannians, analogous to the embeddings of $\SU(n)$, $\Sp(n)$, and the quaternionic Grassmannians we have considered in this section, do \emph{not} satisfy $\ACS<0$.
\end{remark}


\begin{thebibliography}{{Nev}14}

\bibitem[ACS18a]{ACS16}
Lucas Ambrozio, Alessandro Carlotto, and Ben Sharp.
\newblock Comparing the {M}orse index and the first {B}etti number of minimal
  hypersurfaces.
\newblock {\em J. Differential Geom.}, 108(3):379--410, 2018.

\bibitem[ACS18b]{ACS17}
Lucas Ambrozio, Alessandro Carlotto, and Ben Sharp.
\newblock A note on the index of closed minimal hypersurfaces of flat tori.
\newblock {\em Proc. Amer. Math. Soc.}, 146(1):335--344, 2018.

\bibitem[BCO16]{BCO}
J\"urgen Berndt, Sergio Console, and Carlos~Enrique Olmos.
\newblock {\em Submanifolds and holonomy}.
\newblock Monographs and Research Notes in Mathematics. CRC Press, Boca Raton,
  FL, second edition, 2016.

\bibitem[Bes08]{Besse}
Arthur~L. Besse.
\newblock {\em {Einstein manifolds}}.
\newblock {Classics in Mathematics}. Springer-Verlag, Berlin, 2008.
\newblock Reprint of the 1987 edition.

\bibitem[CM16]{ChodoshMaximo16}
Otis Chodosh and Davi Maximo.
\newblock On the topology and index of minimal surfaces.
\newblock {\em J. Differential Geom.}, 104(3):399--418, 2016.

\bibitem[Fer70]{Ferus70}
Dirk Ferus.
\newblock On the type number of hypersurfaces in spaces of constant curvature.
\newblock {\em Math. Ann.}, 187:310--316, 1970.

\bibitem[FKM81]{FerusKarcherMunzner81}
Dirk Ferus, Hermann Karcher, and Hans~Friedrich M{\"u}nzner.
\newblock {Cliffordalgebren und neue isoparametrische {H}yperfl{{\"a}}chen}.
\newblock {\em Math. Z.}, 177(4):479--502, 1981.

\bibitem[GR70]{GromovRohlin70}
M.~L. Gromov and V.~A. Rohlin.
\newblock Imbeddings and immersions in {R}iemannian geometry.
\newblock {\em Uspehi Mat. Nauk}, 25(5 (155)):3--62, 1970.

\bibitem[G{\"u}n89]{Guenther89}
Matthias G{\"u}nther.
\newblock On the perturbation problem associated to isometric embeddings of
  {R}iemannian manifolds.
\newblock {\em Ann. Global Anal. Geom.}, 7(1):69--77, 1989.

\bibitem[MR17]{MendesRadeschi17}
R.~A.~E. {Mendes} and M.~{Radeschi}.
\newblock {Generalized immersions and minimal hypersurfaces in compact
  symmetric spaces}.
\newblock {\em ArXiv e-prints}, August 2017.

\bibitem[{Nev}14]{Neves14}
A.~{Neves}.
\newblock {New applications of Min-max Theory}.
\newblock {\em ArXiv e-prints}, September 2014.

\bibitem[OT75]{OzekiTakeuchi75}
Hideki Ozeki and Masaru Takeuchi.
\newblock {On some types of isoparametric hypersurfaces in spheres. {I}}.
\newblock {\em T{\^o}hoku Math. J. (2)}, 27(4):515--559, 1975.

\bibitem[OT76]{OzekiTakeuchi76}
Hideki Ozeki and Masaru Takeuchi.
\newblock {On some types of isoparametric hypersurfaces in spheres. {II}}.
\newblock {\em T{\^o}hoku Math. J. (2)}, 28(1):7--55, 1976.

\bibitem[Ros06]{Ros06}
Antonio Ros.
\newblock One-sided complete stable minimal surfaces.
\newblock {\em J. Differential Geom.}, 74(1):69--92, 2006.

\bibitem[Sav10]{Savo10}
Alessandro Savo.
\newblock Index bounds for minimal hypersurfaces of the sphere.
\newblock {\em Indiana Univ. Math. J.}, 59(3):823--837, 2010.

\bibitem[Wal72]{Wallach72b}
Nolan~R. Wallach.
\newblock {Minimal immersions of symmetric spaces into spheres}.
\newblock pages 1--40. Pure and Appl. Math., Vol. 8, 1972.

\end{thebibliography}

\def\cprime{$'$}

\end{document}